\documentclass[11pt]{amsart}
\usepackage{amsmath, amsthm, amssymb, latexsym,url}
\usepackage{hyperref}

\author{Joseph Vandehey}
\title{New normality constructions for continued fraction expansions}
\date{\today}

\newtheorem{thm}{Theorem}[section]
\newtheorem{cor}[thm]{Corollary}

\newtheorem{lem}[thm]{Lemma}

\newtheorem{prop}[thm]{Proposition}

\allowdisplaybreaks

\begin{document}

\begin{abstract}
Adler, Keane, and Smorodinsky showed that if one concatenates the finite continued fraction expansions of the sequence of rationals
\[
\frac{1}{2}, \frac{1}{3}, \frac{2}{3}, \frac{1}{4}, \frac{2}{4}, \frac{3}{4}, \frac{1}{5}, \cdots
\]
into an infinite continued fraction expansion, then this new number is normal with respect to the continued fraction expansion. We show a variety of new constructions of continued fraction normal numbers, including one generated by the subsequence of rationals with prime numerators and denominators:
\[
\frac{2}{3}, \frac{2}{5}, \frac{3}{5}, \frac{2}{7}, \frac{3}{7}, \frac{5}{7},\cdots.
\]
\end{abstract}

\maketitle

\section{Introduction}

A number $x\in [0,1)$ is said to be normal (to base $10$) if for any string $s=[d_1,d_2,\dots,d_k]$ of decimal digits, we have
\[
\lim_{N\to \infty} \frac{A_s(N;x)}{N} = \frac{1}{10^k}
\]
where $A_s(N;x)$ is the number of times the string $s$ appears starting in the first $N$ digits of the decimal expansion of $x$. For numbers outside of the interval $[0,1)$, we consider them to be normal if the number taken modulo $1$ is normal. While it is a simple consequence of the pointwise ergodic theorem that almost all real numbers are normal, there is no commonly used irrational number, such as $\pi$, $e$, or even $\sqrt{2}$, that is known to be normal.

However, mathematicians have constructed a wide variety of normal numbers, the first of which was found by Champernowne: he showed that the number \[ 0.123456789101112131415\ldots, \]formed by concatenating all the natural numbers in order, is normal \cite{Champernowne}. Following Champernowne, Besicovitch showed that the number \[0.149162536496481100\ldots,\] formed by taking all the perfect squares in order, is normal \cite{BesicovitchSquares}. These constructions inspired a large area of research, as mathematicians considered for which functions $f(n)$ would the number \[ 0.f(1)f(2)f(3)\dots\] be normal. A related question asks whether just concatenating the prime values of a function, \[0.f(2)f(3)f(5)f(7)f(11)\dots,\] also generates a normal number. A small selection of all the results in this area include the work of Davenport and Erd\H{o}s \cite{DavenportErdos}; Nakai and Shiokawa \cite{NakaiShiokawa1}; De Koninck and Katai \cite{dKK2}; Madritsch, Thuswaldner, and Tichy \cite{MTT}; and the author \cite{VandeheyAdditive}.

Of particular interest to this paper is the work of Copeland and Erd\H{o}s \cite{CopelandErdos}. They showed that almost all integers are $(\epsilon,k)$-normal, which refers to the fact that each string of length $k$ appears in the decimal expansion of the integer to within $\epsilon$ of the expected frequency $10^{-k}$. Thus, in place of the sequence of all positive integers, as in Champernowne, if we take a sufficiently dense subset of the positive integers and concatenate those, we expect to get a number that is normal as well. Copeland and Erd\H{o}s showed that the primes constitute a sufficiently dense subset of the integers, and thus concatenating them into $0.2357111317\dots$ produces a normal number.

The notion of normality extends nicely to continued fraction expansions. Here, despite the wide variety of construcions of  numbers normal to base $10$, there are only two distinct types of constructions that have been discovered for numbers normal with respect to the continued fraction expansion. Let us recall some basic definitions. For a real number $x$, the continued fraction expansion for $x$ is given by
\[
x = a_0+\cfrac{1}{a_1+\cfrac{1}{a_2+\cfrac{1}{a_3+\dots}}},
\]
where $a_0\in \mathbb{Z}$ and $a_i \in \mathbb{N}$ for $i\ge 1$. If $x$ is rational, then there are two possible expansions; for example, one could end in a $5$ while the other could end in a $4$ followed by a $1$. When there is ambiguity as to the digits of a rational number, we will assume throughout the rest of the paper that we are always taking the longest possible finite expansion . If $x$ is irrational, this expansion is unique and infinite. We often shorthand this notation by writing $x=\langle a_0; a_1,a_2,a_3,\dots\rangle $, or, if $a_0=0$, by $x=\langle a_1,a_2,\dots\rangle $.

The quantity $a_i(x)$ refers to the $i$th digit of $x$, if it exists. For a given expansion $\langle a_0; a_1,a_2,a_3,\dots\rangle$ the truncated expansions $\langle a_0; a_1,a_2,\dots,a_k\rangle$ are known as the convergents and are represented by the rational number $p_k/q_k$ in lowest terms.

Given a string $s=[d_1,d_2,\dots,d_k]$ of natural numbers, we let \[
C_s = \{x\in [0,1): a_i(x)=d_i, 1\le i \le k\}
\]
be the cylinder set corresponding to $s$.  Finally we also have the Gauss measure, $\mu$, which for a Lebesgue-measurable set $A\subset [0,1)$ is defined by
\[
\mu(A) = \frac{1}{\log 2} \int_A \frac{1}{1+x} \ dx.
\]

With these definitions, we say a number $x$ is continued fraction normal (or just CF-normal) if, for any finite string $s$ of natural numbers, we have
\[
\lim_{N\to \infty} \frac{A_s(N;x)}{N} = \mu(C_s),
\]
with an analogous extension to all real numbers. Here (and in the remainder of this paper) $A_s(N;x)$ is the number of times the string $s$ occurs starting in the first $N$ continued fraction digits of $x$. (For clarity, we mean that $s$ appears in the expansion of $x$ starting somewhere between the digits $a_1$ and $a_N$, inclusive.) As with decimal expansions, it is possible to show that almost all real numbers are CF-normal by using the pointwise ergodic theorem. 

The first example of a CF-normal number appears to have been given by Postnikov and Pyateckii \cite{PP}. They constructed a series of very long, but finite length strings $X_i$ with good small-scale normality properties---that is, the frequency for which all sufficiently short strings $s$ appeared in $X_i$ was close to the desired asymptotic frequency $\mu(C_s)$---with each successive $X_i$ having better and better small-scale normality properties, approximating even more strings to an even better amount. To produce their desired CF-normal number $x$, they concatenated the strings $X_i$ in succession. (This technique is  generalizable to many, many other systems, as demonstrated by Madritsch and Mance \cite{MaMa}, although curiously, they seem to have been unaware of Postnikov and Pyateckii's work.) Unfortunately, the computation of the strings $X_i$ is not nearly as elegant as Champernowne's simple construction.

For elegance, we turn to Adler, Keane, and Smorodinsky \cite{AKS}. They considered the simple sequence of rational numbers given by
\[
\frac{1}{2}, \frac{1}{3}, \frac{2}{3}, \frac{1}{4}, \frac{2}{4}, \frac{3}{4}, \cdots.
\]
They showed that if one concatenates the finite continued fraction expansions of these rational numbers in order, then the resulting infinite continued fraction is CF-normal. More precisely, the finite continued fraction expansions are $\langle 2 \rangle $, $\langle 3 \rangle$, $\langle 1,2 \rangle$, $\langle 4 \rangle$, $\langle 2\rangle$, $\langle 1,3\rangle$ and so on, so the concatenation gives $\langle 2,3,1,2,4,2,1,3,\dots \rangle$.\footnote{One could choose either of the two finite expansions for each rational and still obtain a CF-normal number. We use the short expansion for readibility.} In fact, their result is slightly stronger than this. Suppose we denote the $n$th rational in this sequence by $r_n$ and let $S\subset \mathbb{N}$ be a set such that
\[
\lim_{N\to \infty} \frac{\#\{n\le N; n \in S\}}{N} = 1,
\]
that is, the set $S$ has asymptotic density $1$. (If this limit exists and equals $\rho$, then we say the set $S$ has asymptotic density $\rho$.) Then if one concatenates the continued fraction expansions of all $r_n$ with $n\in S$,  one obtains a CF normal number. However, this is not strong enough to even remove the ``duplicated" fractions, such as $2/4$ which already appeared as $1/2$, since the asymptotic density of the corresponding set is $6/\pi^2$. (They remark that these duplicated fractions could be removed, and it is likely they could by using visible point estimates or the more careful asymptotic estimates we use in this paper, but these details are not included.)

The proof of Adler, Keane, and Smorodinsky is somewhat similar to that of Copeland and Erd\H{o}s. They show that most rationals with denominator at most $m$ are likely to have good small-scale normality properties---some equivalent of $(\epsilon,k)$-normality---and thus if one concatenates the continued fraction expansions of all the rationals with denominator at most $m$, they should obtain a number that is close to being normal, in some sense. As part of their proof, Adler, Keane, and Smorodinsky use an ergodic theorem to prove that the measure of a particular sequence of sets approaches one, but as is standard for ergodic results, the rate of convergence is not clear.

In this paper, we will use a metrical result, based on work of Philipp \cite{Philipp}, to get better asymptotics on how many rationals with denominator at most $m$ have good small-scale normality properties. This allows us to prove a variety of new constructions. The following theorem generalizes the work of Adler, Keane, and Smorodinsky and gives a continued fraction analogue of the ``combinatorial method of normal number proofs" (see Theorem 1 of Pollack and Vandehey \cite{PV}). We remark that all asymptotic notations used in this paper will be defined at the end of the introduction.

\begin{thm}\label{thm:main}
Let $\{r_i\}_{i=1}^\infty$ denote the sequence of all rational numbers (in lowest terms) in the interval $(0,1)$, ordered in the following way:
\[
r_1=\frac{1}{2}, \ r_2=\frac{1}{3}, \  r_3=\frac{2}{3}, \ r_4=\frac{1}{4}, \ r_5=\frac{3}{4}, \cdots.
\]

Let $f:\mathbb{N}\to \mathbb{N}$, and define the number $x_f$ as the number constructed by concatenating the continued fraction expansions of the rationals $r_{f(1)}, r_{f(2)}, r_{f(3)},\cdots$.

Let $L(r)$ denote the length of the continued fraction expansion of $r$. Suppose that
\[
N= o\left( \sum_{n=1}^N L(r_{f(n)}) \right) \quad \text{ and } \quad N\cdot \max_{1\le n \le N} L(r_{f(n)}) = O\left( \sum_{n=1}^N L(r_{f(n)})\right)
\]
and that for any set $S\subset \mathbb{N}$ that satisfies $\#\{n\in S: n \le x\}  = O(x/\log x)$, we have that $f^{-1}(S)$ has asymptotic density $0$. Then $x_f$ is CF normal.

\end{thm}

Although Theorem \ref{thm:main} is stated in a very general fashion, we will be interested in this paper primarily in cases where we consider subsequences of the sequence of rationals considered by Adler, Keane, and Smorodinsky. The following corollary follows from Theorem \ref{thm:main}.

\begin{cor}\label{cor:main}
Let $i_1<i_2<i_3<\dots$ be an infinite, increasing subsequence of $\mathbb{N}$. Let $R$ denote the subsequence of $\{r_i\}_{i=1}^\infty$ considered in Theorem \ref{thm:main}, given by $\{r_{i_j}\}_{j=1}^\infty$. Suppose that if $R(m)$ denotes the set of $p/q\in R$ in lowest terms with $q\le m$, then
\[
\frac{m^2}{\log m} = o ( |R(m)|), \ m\to \infty.
\]
Then the number formed by concatenating the continued fraction expansions of the rationals $r_{i_1},r_{i_2},r_{i_3},\dots$ in order is CF-normal.
\end{cor}

\begin{cor}\label{cor:squarefree}
If one concatenates the sequence of rationals $r_i$ that are in lowest terms with squarefree numerator and denominator, then the resulting number is CF-normal.
\end{cor}

Corollary \ref{cor:squarefree} follows from the fact that the sets $R(m)$ will have size on the order of $m^2$.

However, Corollary \ref{cor:main} is not strong enough even to consider the subsequence composed of all numerators but only prime denominators. (In this case, we would have $|R(m)| = O(m^2/\log m)$.) For this we must prove new results.

\begin{thm}\label{thm:primes}Let $\mathbb{P}$ denote the set of primes.
Let $R$ be one of the following subsequences of the rational numbers $\{r_i\}_{i=1}^\infty$ considered in Theorem \ref{thm:main}: 
\begin{enumerate}
\item the subsequence whose numerators are in $\mathbb{N}$ and whose denominators are in $\mathbb{P}$; 
\item the subsequence whose numerators are in $\mathbb{P}$ and whose denominators are in $\mathbb{N}$;  or,
\item the subsequence whose numerators and denominators are in $\mathbb{P}$.
\end{enumerate}
If the indices of the $r_i\in R$ are, in increasing order, $i_1,i_2,i_3,\dots$, then the number formed by concatenating the continued fraction expansions of the rationals $r_{i_1},r_{i_2},r_{i_3},\dots$ in order is CF-normal.
\end{thm}

Although we will not make it precise here, the statement of Theorem \ref{thm:primes} can be improved a fair amount. For example, we could replace $\mathbb{P}$ by any sufficiently dense subset of the primes, such as the set of primes congruent to $1 $ modulo $4$.

We close the introduction with an open problem. The results of this paper, when compared to the prior work of Adler, Keane, and Smorodinsky, mimic how Copeland and Erd\H{o}s extended the work of Champernowne, by replacing the set of positive integers with the set of primes. So we ask: can an analogy of Besicovitch's work be proven in the continued fraction case---that  is, can one form a CF-normal number by concatenating those rationals whose numerators and denominators are perfect squares? This may be possible by breaking the continued fraction expansion of such rationals into two pieces, each of which has denominator around size $m$.

We will make frequent use of asymptotic notations in this paper. By $f(x)=O(g(x))$, equivalently $f(x) \ll g(x)$, we mean that there exists some constant $C$, called the implicit constant, such that $|f(x)|\le C |g(x)|$. By $f(x) \asymp g(x)$, or $f(x)$ is on the order of $g(x)$, we mean that $f(x)= O(g(x))$ and $g(x)=O(f(x))$. By $f(x)=o(g(x))$ we mean that $\lim_{x\to \infty} f(x)/g(x) = 0$. By $f(x) \sim g(x)$, we mean that $f(x) = g(x) (1+o(1))$.

\section{Metrical results}

In this section we will provide metric results which will be required to prove our main results. Many of these results are variants of work of Philipp \cite{Philipp}. We could cite Philipp's results directly, but there are some spots where he glosses over some complicated calculations, and we provide them here.

First we recall some elementary facts about continued fractions. For references, see \cite{KhinchinBook} or  sections 1.3 of \cite{DK}. 

Let $T:[0,1)\to [0,1)$ be the standard Gauss map given by
\[
Tx = \begin{cases} \frac{1}{x} - \lfloor \frac{1}{x} \rfloor, & x\neq 0,\\
0, & x=0.
\end{cases}
\]
This acts as a forward shift on continued fraction expansions, so that $T\langle a_1,a_2,a_3,\dots\rangle =\langle a_2,a_3,\dots\rangle$. The map $T$ leaves the Gauss measure $\mu$ invariant, that is, for any measurable set $A$, we have $\mu(A)= \mu(T^{-1} A)$.

If $x=\langle a_1,a_2,a_3,\dots\rangle$ with $n$th convergent $p_n/q_n = \langle a_1,a_2,\dots, a_n\rangle$, then $q_{n-1}/q_n = \langle a_n,a_{n-1},\dots,a_1\rangle$ and $p_{n-1}/p_n = \langle a_n,a_{n-1}, \dots, a_2\rangle$.  Given a string $s=[d_1,d_2,d_3,\dots,d_k]$, the cylinder set $C_s$ consists of all points between \begin{equation}\label{eq:endpoints} \frac{p_n}{q_n}= \langle d_1,d_2,\dots,d_k\rangle \text{ and } \frac{p_n+p_{n-1}}{q_n+q_{n-1}} = \langle d_1,d_2,\dots, d_k+1\rangle. \end{equation}
Thus the measure of the cylinder set can be calculated to be
\[
\mu(C_s) = \frac{1}{\log 2} \left| \log \frac{(p_n+q_n)(q_{n-1}+q_n)}{q_n(p_{n-1}+p_n+q_{n-1}+q_n)} \right|.
\]
This is invariant if we swap $p_n$ and $q_{n-1}$ and thus also left invariant if we swap $s$ for $\overline{s}=[d_k,d_{k-1},d_{k-2},\dots,d_1]$. 

Given two rational numbers $r=p/q$ and $r'=v/u$ in lowest terms, concatenating  their continued fraction expansions gives the rational number
\[
\frac{up+vp'}{uq+vq'}
\]
where $p'/q'$ is the rational number obtained by removing the last continued fraction digit from the expansion of $r$. (Note: the result of the concatenation will depend on which of the two expansions we choose for $r$.)

The denominators of the convergents also satisfy a recurrence relation: $q_0 = 1$, $q_1=a_1$ and $q_n=a_n q_{n-1}+q_{n-2}$ for $n\ge 2$. The numerators satisfy a similar recurrence. From these it can be shown that $|p_nq_{n-1} - p_{n-1}q_n |= 1$. These facts, together with \eqref{eq:endpoints}, imply that the Lebesgue measure of a cylinder set $C_s$ is given by $q_n^{-1}(q_n+q_{n-1})^{-1}$. In addition, since $a_n\ge 1$ for all $n\ge 1$, the recurrence relation implies that $q_n$ must be of the size of the $n$th Fibonacci number, i.e.,
\begin{equation}\label{eq:equation}
q_n \gg G^n
\end{equation} where $G=(1+\sqrt{5})/2$. These facts together imply that $\lambda(C_s) =O( 2^{-k})$ if $s$ has length $k$ (in which case we say that $C_s$ is a rank $k$ cylinder).

For the remainder of this section let $m\ge 3$ be an integer, and let $g=\pi^2/(12\log 2)$ be the logarithm of the Khinchin-L\'{e}vy constant. Let $\delta, \eta$ be positive real numbers with $\delta<1/3$. We will also let $s=[d_1,d_2,\dots,d_k]$ be a finite, nonempty string of positive integers.

We note that $g$ can be written as
\begin{equation}\label{eq:gformulation}
g= \int_0^1 -\log x \ d\mu(x) = \frac{1}{\log 2} \int_0^1 \frac{-\log x}{1+x} \ dx.
\end{equation}

By classical ergodic results (see, for example, section 3.5 of \cite{DK}), we have that for almost all $x$, that 
\begin{equation}\label{eq:qnexpected}
\lim_{N\to \infty} \frac{\log q_N(x)}{N} = g,
\end{equation}
where $q_N(x)$ denotes the denominator of the $N$th convergent to $x$,
and
\begin{equation}\label{eq:sexpected}
\lim_{N\to \infty} \frac{A_s(N;x)}{N} = \mu(C_s).
\end{equation}
The formulas \eqref{eq:qnexpected} and \eqref{eq:sexpected} tell us the expected behavior of  continued fraction expansions. 

For the following definition, let  $n=n_\delta= \lfloor (1-2\delta)(\log m)/g\rfloor$, so that we expect rational numbers with denominator $m$ to have slightly more than $n$ continued fraction digits. As before, we take $L(r)$ to be the number of continued fraction digits of $r$, ignoring any $a_0$ digit. We define $\Gamma_{m,\delta,s,\eta}$ to be the set of all fractions $r=\langle a_1,a_2,a_3,\dots,a_{L(r)}\rangle $ in the interval $(0,1)$ that satisfy the following conditions
\begin{enumerate}
\item If $r$ is written in lowest terms, then the denominator of $r$ is at most $m$.
\item Either $L(r)< n$ or  \[\left|\frac{\log q_{n}}{n} - g\right|>\delta\] or
\[
\left| \frac{\#\{0 \le i \le n -k: [a_{i+1},a_{i+2},\dots,a_{i+k}]=s\}}{n}- \mu(C_s) \right| >\eta
\]
\end{enumerate} 
The fractions in $\Gamma_{m,\delta,s,\eta}$ are numbers which have somewhat unusual properties, either their continued fraction expansions are very short or their $n$th convergent falls away from the expected behavior. The main result of this section is the following, which suggests that these numbers are rather rare:

\begin{prop}\label{prop:mainestimate}
For fixed $\delta,s,\eta$ and sufficiently large $m$, we have $\Gamma_{m,\delta,s,\eta} = O(m^2/\log m)$.
\end{prop}

To prove this we will need a series of lemmas, whose proofs will take up the bulk of this section.

\begin{lem}\label{lem:Philipp}
Let $C_s$ be a cylinder set of rank $k\ge 1$ and let $F$ be any measurable subset of $[0,1)$. Then for $n\ge 0$
\[
\mu(C_s\cap T^{-n-k}F )= \mu(C_s)\mu(F)(1+O(\tau^{\sqrt{n}}))
\]
where $0<\tau<1$ is a fixed constant and the implicit constant is uniform over all $s$ and $F$.
\end{lem}

This is just Lemma 2 in \cite{Philipp}, so we omit the proof.

Since the implicit constant in Lemma \ref{lem:Philipp} is uniform, it is clear that we may replace $C_s$ by any disjoint union of rank $k$ cylinders, and the statement would still be true.

\begin{prop}\label{prop:goodcount}
Let $s=[d_1,d_2,\dots,d_k]$ be a string and let $A_s(N;x)$ again 
denote the number of times this string occurs starting in the first $N$ positions of a real number $x$. 

Let $E_{\epsilon,s,N}$ denote the set of $x\in[0,1)$ such that
\[
|A_s(N;x)-\mu(C_s)N | >\epsilon \mu(C_s) N
\]
Then
\[
\mu(E_{\epsilon,s,N}) =  O\left( \frac{1}{\epsilon^2 \mu(C_s) N}\right)
\]
where the implicit constant is uniform over all $\epsilon$, $s$, and $N$, but may  depend on $k$. Moreover, $E_{\epsilon,s,N}$ can be expressed as the union of rank $N+k-1$ cylinders.
\end{prop}

We note that $A_x(N;x)$ is only meaningful if $x$ has at least $N+k-1$ continued fraction digits. Thus rational numbers with fewer than that many digits are not included in $E_{\epsilon,s,N}$ by default.

The statement of the theorem, and the subsequent proof, are very similar to  Theorem 3 in \cite{Philipp}; however, we have made this statement uniform in $x$.

\begin{proof}
Let $I_s(x)$ denote the characteristic function of $C_s$. Since $T$ acts as a forward shift on the digits, we can write $A_s(N;x)=\sum_{n=0}^{N-1} I_s(T^n x)$. Also, since
\[
\int_0^1 I_s(T^n x) \ dx= \mu(T^{-n}C_s) = \mu(C_s),
\]
we have that
\[
\int_0^1 A_s(N;x) \ dx = \mu(C_s) N.
\]

Then we have,
\begin{align*}
\int_0^1 \left( A_s(N;x)-\mu(C_s)N \right)^2 \ d\mu&= \sum_{0\le i,j< N}\left( \int_0^1  I_s(T^i x) I_s(T^j x) \ d\mu \right) - \mu(C_s)^2 N^2\\
&=\Sigma_1+\Sigma_2+\Sigma_3-\mu(C_s)^2N^2,
\end{align*}
where $\Sigma_1$ is twice the sum running over $0\le i <j<N $ with $j-i\le k$, $\Sigma_2$ is twice the sum running over $0\le i<j<N$ with $j-i>k$, and $\Sigma_3$ is the sum running over $i=j$.

For each term in $\Sigma_1$, we have 
\[
\int_0^1  I_s(T^i x) I_s(T^j x) \ d\mu = \mu(T^{-i}C_s \cap T^{-j} C_s) \le \mu(T^{-i} C_s) = \mu(C_s),
\]
so the sum over all such $i$ and $j$ satisfying the conditions of $\Sigma_1$ is bounded by $2k \mu(C_s) N$. 

We bound $\Sigma_2$ using Lemma \ref{lem:Philipp}, noting that $T^{-i} C_s$ can be expressed as a disjoint union of rank $i+k$ cylinders:
\begin{align*}
&2\sum_{j< N} \sum_{0\le i<j-k} \int_0^1  I_s(T^i x) I_s(T^j x) \ d\mu\\
& \qquad= 2\sum_{j< N} \sum_{0\le i<j-k}\mu(T^{-i}C_s \cap T^{-j} C_s) \\
& \qquad= 2 \sum_{j< N} \sum_{0\le i<j-k} \mu(T^{-i}C_s) \mu(T^{-j}C_s)\left(1+O\left(\tau^{\sqrt{j-i-k}}\right)\right)\\
& \qquad= 2\sum_{j< N} \sum_{0\le i<j-k} \mu(C_s)^2 \left(1+O\left( \tau^{\sqrt{j-i-k}}\right)\right)\\
& \qquad= \mu(C_s)^2 N(N-2k+1)+ O\left( \mu(C_s)^2 \sum_{j\le N} \sum_{i<j-k} \tau^{\sqrt{j-i-k}} \right)\\
& \qquad= \mu(C_s)^2 N^2 + O\left( \mu(C_s)N\right),
\end{align*}
where this final step derives from the fact that $\mu(C_s)^2 \le \mu(C_s)$ and the following work, substituting the variable $\ell$ for $j-i$,
\begin{equation}\label{eq:qexpbound}
\sum_{j< N} \sum_{0\le i<j-k} \tau^{\sqrt{j-i-k}} = \sum_{k< \ell< N} (N-\ell )\tau^{\sqrt{\ell-k}} \le N  \sum_{k< \ell\le N} \tau^{\sqrt{\ell-k}} = O(N).
\end{equation}
Throughout this paragraph, the implicit constant is only dependent on $k$.

As the square of any characteristic function is itself, we have that $\Sigma_3$ is just $\mu(C_s) N$. 

By combining these estimates together, we have
\[
\int_0^1 \left( A_s(N;x)-\mu(C_s)N \right)^2 \ d\mu = O\left( \mu(C_s) N\right).
\]

Thus,
\[
\mu(E_{\epsilon,s,N})\le \int_{E_{\epsilon,s,N}}\dfrac{ \left( A_s(N;x)-\mu(C_s)N \right)^2}{\epsilon^2 \mu(C_s)^2 N^2} \ d\mu= O\left(\frac{1}{\epsilon^2 \mu(C_s)N} \right)
\]
which gives the desired result. The fact that $E_{\epsilon,s,N}$ can be written as a union of rank $N+k-1$ cylinders comes from the fact that the value of $A_s(N;x)$ depends only on which $N+k-1$ rank cylinder $x$ lies in.
\end{proof}

\begin{prop}\label{prop:gooddenom}
Let $q_n(x)$ denote the denominator of the $n$th convergent to $x$. Let $\epsilon>0$ be a fixed constant, and let $F_{\epsilon,N}$ denote the set of $x\in[0,1)$ such that
\[
\left|\frac{\log q_N(x)}{N} - g \right| > \epsilon .
\]

Then  $\mu(F_{\epsilon,N}) = O(1/N)$ with the implicit constant dependent only on $\epsilon$. Moreover, $F_{\epsilon,N}$ can be expressed as a disjoint union of rank $N$ cylinders.
\end{prop}

\begin{proof}
Let $f_n(x)$ be a function given by $\langle a_n(x);a_{n-1}(x),\dots,a_1(x)\rangle$ if $x=\langle a_1(x),a_2(x),\dots\rangle$ is irrational and $f_n(x)=0$ if $x$ is rational. It follows that $f_n(x) = q_n(x) / q_{n-1}(x)$ if $x$ is irrational. We always take $q_0(x)=1$. Thus, we have that
\begin{equation}\label{eq:qNtofi}
\frac{\log q_N(x)}{N} = \frac{1}{N} \sum_{i=1}^N \log f_i(x).
\end{equation}

We will also define $f_n^{(k)}(x) = f_k(T^{n-k} x)$ when $n \ge k$. For irrational $x$, the function $f_n^{(k)}$ truncates $f_n$ after $k$ digits, thus we expect it to be a good approximation to $f_n$. In fact, we have that
\[
\left| \log f_n(x) - \log f_n^{(k)}(x) \right| \le |f_n(x) - f_n^{(k)}(x) | ,
\]
but (ignoring the integer parts, which cancel) this is just the distance between two points in the same rank $k-1$ cylinder set, which, as we mentioned at the start of this section, is at most $O(2^{-k})$, thus
\begin{equation}\label{eq:philippeq1}
\left| \log f_n(x) - \log f_n^{(k)}(x) \right| = O(2^{-k})
\end{equation}
for all irrational $x\in [0,1)$.

We  let 
\[
\lambda_k = \int_0^1 \log f_k(x) \ d \mu.
\]
Clearly $f_k(x)$ is fixed on any rank $k$ cylinder. In fact, let us define a function on strings $s=[d_1,d_2,\dots,d_k]$ given by $g(s)= \log(\langle d_k; d_{k-1},\dots,d_1\rangle)$. Then if $x\in C_s$ with $s$ having length $k$, then $f_k(x) = g(s)$. From the start of this section, we saw that if $s$ is a string and $\overline{s}$ is this string in reverse, then $\mu(C_s)=\mu(C_{\overline{s}})$.  These facts give the following:
\begin{align*}
\lambda_k &= \int_0^1 \log( \langle a_k(x); a_{k-1}(x) \dots, a_1(x) \rangle ) \ d\mu(x) =\sum_{s} g(s) \mu(C_s)\\
&= \sum_s g(s)\mu(C_{\overline{s}}) = \int_0^1 \log ( \langle a_1(x) ; a_2(x), \dots a_k(x) \rangle) \ d\mu(x),
\end{align*}
where here the sums over $s$ always go over all strings $s$ with $k$ digits.
However the point \[ \frac{1}{x} - a_1(x)\]  and the point $\langle  a_2(x),\dots a_k(x)\rangle$ belong to the same rank $k-1$ cylinder set, thus, by our earlier argument, are within $O(2^{-k})$ of each other. Thus, by \eqref{eq:gformulation} we have
\begin{align*}
\lambda_k &= \int_0^1 \log 1/x \ d\mu(x) +  \int_0^1  \log \langle a_1(x);a_2(x),\dots,a_k(x) \rangle - \log 1/x \ d\mu(x) \\
 &=  g +  O\left( \int_0^1\left| \frac{1}{x} - \langle a_1(x);a_2(x), \dots, a_k(x) \rangle\right| \ d\mu(x) \right) \\
 &=  g +  O\left( \int_0^1 2^{-k} \ d\mu(x) \right) \\
&= g + O(2^{-k}).
\end{align*}


Now we let $N=n+k$ for some choice of positive integers $n$ and $k$ to be made later. (We note, for clarity, that the $k$ and $n$ here do not relate to the length of any cylinder set, nor to the constant $n_\delta$ defined earlier. They are wholly separate variables.) The assumption will be that $k$ is fixed so that as $N$ varies, so does $n$. Then, by using \eqref{eq:qNtofi}, \eqref{eq:philippeq1}, and the last part of the previous paragraph, 
 we have
\begin{align*}
&\left| \frac{\log q_N(x)}{N} - g \right| =\left| \frac{1}{N} \sum_{i=1}^N \log f_i(x) - g\right| \\
&\qquad < \frac{N}{n} \left| \frac{1}{N} \sum_{i=1}^N \log f_i(x) - g\right| \le \frac{k}{n} g +\left| \frac{1}{n} \sum_{i=1}^N \log f_i(x) -g\right|\\
& \qquad < \frac{k}{n} g + \left| \frac{1}{n} \sum_{i=1}^k \log f_i(x) \right| + \frac{1}{n} \sum_{i=k+1}^{n+k} \left| \log f_i(x)-\log f_i^{(k)}(x)\right| \\
&\qquad \qquad + \left| \frac{1}{n} \sum_{i=k+1}^{n+k} \log f_i^{(k)}(x) - \lambda_k \right| + \left| \lambda_k -g\right|\\
&\qquad = O\left( \frac{k}{n}\right) + \left| \frac{1}{n} \sum_{i=1}^k \log f_i(x) \right|  + \left| \frac{1}{n} \sum_{i=k+1}^{n+k} \log f_i^{(k)}(x) - \lambda_k \right|  + O\left( 2^{-k} \right). 
\end{align*}
If we suppose that $k$ is a fixed sufficiently large integer so that the $O(2^{-k})$ term will be less than $\epsilon/4$, then the desired result will clearly hold provided we can show that the set of $x$ for which
\begin{equation}\label{eq:firstineq}
 \left| \frac{1}{n} \sum_{i=1}^k \log f_i(x) \right| > \frac{\epsilon}{4}
\end{equation}
or
\begin{equation}\label{eq:secondineq}
 \left| \frac{1}{n} \sum_{i=k+1}^{n+k} \log f_i^{(k)}(x) - \lambda_k \right| > \frac{\epsilon}{4}
\end{equation}
has size at most $O(1/N)$ for sufficiently large $N$. In fact, since we have assumed $k$ is fixed, dependent on $\epsilon$, it suffices to show that the set has size $O(1/n)$.

We will need that for any non-negative integers $i,j$ we have 
\begin{equation}\label{eq:twotermbound}
\int_0^1 \log f_i(x) \log f_j(x) \ d \mu \le M
\end{equation}
for some uniform constant $M$. To see this, first note that $f_i(x) \le a_i(x)+1$. We also have for any fixed integer $a$ that
\[
\mu(C_{[a]}) = \frac{1}{\log 2} \log \left( 1+ \frac{1}{a(a+2)} \right) \le \frac{L}{a^2}
\]
for some sufficiently large, uniform constant $L$. Therefore, by Lemma \ref{lem:Philipp},  if $i\neq j$ we have
\begin{align*}
\int_0^1 \log f_i(x) \log f_j(x) \ d\mu &\le \sum_{a,b=1}^\infty \log (a+1)\log(b+1) \mu(\{x\in[0,1): a_i(x)=a,a_j(x)=b\})\\
&= \sum_{a,b=1}^\infty \log(a+1)\log (b+1)\mu(T^{-i} C_{[a]} \cap T^{-j} C_{[b]})\\
&\le \left( \sum_{a,b=1}^\infty \frac{{L}^2\log(a+1)\log(b+1)}{a^2b^2}\right) \left(1+O(q^{\sqrt{|j-i|-1}})\right),
\end{align*}
and this is uniformly bounded as the sum here converges. Likewise, if $i=j$, we get
\begin{align*}
\int_0^1 \log f_i(x) \log f_j(x) \ d\mu \le \sum_{a=1}^\infty \frac{{L}^2\log(a+1)^2}{a^2},
\end{align*}
which is also uniformly bounded.

By using \eqref{eq:twotermbound}, we have
\[
\int_0^1\left( \frac{1}{n} \sum_{i=1}^k \log f_i(x) \right)^2 \ d\mu= O\left( \frac{k^2}{n^2} \right)
\]
So by a similar argument to the final step in the proof of Proposition \ref{prop:goodcount}, we have that \eqref{eq:firstineq} holds on a set of $\mu$-measure $O(k^2/n^2\epsilon^2)=O_{\epsilon}(1/n^2)$.

Proving that \eqref{eq:secondineq} holds on a set of the desired size shall require a few more steps and take up the remainder of the proof. 

First recall that $f_i^{(k)}(x) = f_k(T^{i-k} x)$. Thus we have that $f_i^{(k)}(x) = g(s)$ if and only if $x\in T^{-(i-k)}C_s$ and $s$ has length $k$. Since $T$ preserves $\mu$, we have
\[
\int_0^1 \log f_i^{(k)}(x)  \ d\mu =\sum_s g(s) \mu(T^{-(i-k)}C_s)= \sum_s g(s) \mu(C_s) = \int_0^1 \log f_k(x) \ d \mu = \lambda_k,
\]
where here again the sums run over all strings $s$ with length $k$.

For $k< i \le j$, let $I(i,j)$ denote
\[
\int_0^1 \log f_i^{(k)}(x) \log f_j^{(k)}(x) \ d\mu.
\]
If $j-i\le k$, then we use Cauchy-Schwarz and \eqref{eq:twotermbound} to bound $I(i,j)$ as follows:
\begin{align*}
I(i,j)&\le  \left( \int_0^1(\log  f_i^{(k)}(x))^2 \ d \mu\right)^{1/2}  \left( \int_0^1 (\log f_j^{(k)} (x))^2\ d \mu\right)^{1/2}\\ 
&= \left( \int_0^1(\log  f_k(x))^2 \ d \mu\right)^{1/2}  \left( \int_0^1 (\log f_k (x))^2\ d \mu\right)^{1/2}\\
&\le C.
\end{align*}
If $j-i>k$, then
\[
I(i,j) = \sum_{s,s'} g(s)\cdot g(s') \cdot \mu(T^{-i} C_s\cap T^{-j}C_{s'})= \sum_{s,s'} g(s)\cdot g(s') \cdot \mu(C_s\cap T^{-(j-i)}C_{s'})
\]
where the sum runs over all $s$ and $s'$ with length $k$. 
By applying Lemma \ref{lem:Philipp} we obtain
\begin{align*}
I(i,j) &=\left( \sum_{s} g(s) \mu(C_s)\right)\left( \sum_{s'} g(s') \mu(C_{s'})\right) \left(1+O\left(\tau^{\sqrt{j-i-k}}\right)\right)\\
&=\left( \int_0^1 \log f_k(x) \ d\mu\right)\left( \int_0^1 \log f_k(x) \ d\mu\right)\cdot  \left(1+O\left(\tau^{\sqrt{j-i-k}}\right)\right)\\
&= \lambda_k^2  \left(1+O\left(\tau^{\sqrt{j-i-k}}\right)\right).
\end{align*}
Similar bounds on $I(i,j)$ hold when $j< i$. 

Thus, by making use of our bound on $I(i,j)$ and \eqref{eq:qexpbound}, we have
\begin{align*}
\int_0^1 \left( \frac{1}{n} \sum_{i=k+1}^{n+k} \log f_i^{(k)}(x) - \lambda_k\right)^2 d \mu& =\frac{1}{n^2} \sum_{i,j= k+1}^n \left( I(i,j)-\lambda_k^2\right)\\
& =  O\left( \frac{1}{n^2}\lambda_k^2\sum_{\substack{i,j=k+1\\j-i>k}}^n q^{\sqrt{j-i-k}}\right) +O\left( \frac{k}{n} \right)\\
&  = O\left( \frac{k}{n}\right).
\end{align*}
Therefore, \eqref{eq:secondineq} holds on a set of size $O(k/n\epsilon^2)$, as desired.
\end{proof}

Now we prove the main result of this section.

\begin{proof}[Proof of Proposition \ref{prop:mainestimate}] 

As in the definition of $\Gamma_{m,\delta,s,\eta}$, we let $n= \lfloor (1-2\delta)(\log m)/g\rfloor$.

Let $\Gamma'_m=\Gamma'_{m,\delta,s,\eta}$ denote the union of cylinder sets $C_\frak{s}$, where the strings $\frak{s}=[a_1,a_2,\dots,\\ a_n]$ have length $n$ and satisfy
$|\frac{\log q_n}{n} - g|\le\delta$---where here and in the sequel $q_n$ is the denominator of the rational number $\langle a_1,a_2,\dots,a_n\rangle $ corresponding to $\frak{s}$---$|\frac{\log q_{n-1}}{n-1} - g|\le\delta/12$, and
\[
\left| \frac{\#\{0 \le i \le n-k: [a_{i+1},a_{i+2},\dots,a_{i+k}]=s\}}{n}- \mu(C_\frak{s}) \right| \le\eta.
\]
Note that with these conditions, we have
\[
e^{(n-1)(g-\delta/12)} \le q_{n-1} < q_n \le e^{n(g+\delta)},
\]
which can, after some algebraic simplification and simple estimations, be shown to imply that
\begin{equation}\label{eq:qnboundedbym}
m^{(1-2\delta)(1-\delta/12g)}\cdot e^{-2g} \le q_{n-1} < q_n \le m^{(1-2\delta)(1+\delta/g)}.
\end{equation}

We pause to emphasize the differences between $\Gamma_{m,\delta,s,\eta}$ and $\Gamma'_m$. The set $\Gamma_{m,\delta,s,\eta}$ is a finite set of rational numbers all of which do not exhibit the expected ergodic behavior, while $\Gamma_{m,\delta,s,\eta}$ is an infinite set of rational and irrational numbers, all of which do exhibit expected ergodic behavior in their first $n$ digits. In fact, if $r$ is a rational number in $\Gamma_{m,\delta,s,\eta}$, then $r \in {\Gamma'_m}^{c}$.

By Propositions \ref{prop:goodcount} and \ref{prop:gooddenom}, we know that $\mu( {\Gamma'_m}^c )= O(1/n) = O(1/\log m)$, where here we now assume that the implicit constant is dependent on the various variables, $\delta$, $\eta$, and $s$. Since the densities of the Gauss measure and Lebesgue measure are within a constant multiple of one another, this statement is also true for the Lebesgue measure. Recall that the Lebesgue measure of the set $C_\frak{s}$ is given by $1/q_n(q_n+q_{n-1})$. Thus
\[
\lambda( {\Gamma'}_{m} ) = \sum_{C_\frak{s}} \frac{1}{q_n(q_n+q_{n-1})},
\]
where the sum runs over all the rank $n$ cylinder sets whose union is $\Gamma'_m$.

Suppose $C_\frak{s}$ is one such cylinder set. How many fractions in lowest terms with denimonator at most $m$ are in this set? Recall that if $\frak{s}=[a_1,a_2,\dots,a_n]$, then any rational number whose continued fraction expansion is given by $\langle a_1,a_2,\dots,a_n,a_{n+1},\dots,a_\ell\rangle$ is equal to 
\[
\frac{u p_n+vp_{n-1}}{uq_n+vq_{n-1}},
\]
where $p_n/q_n= \langle a_1,\dots,a_n\rangle $, $p_{n-1}/q_{n-1}=\langle a_1,\dots,a_{n-1}\rangle $, and $v/u=\langle a_{n+1},a_{n+2},\dots,a_\ell\rangle$ with $1\le v<u$. (Note: our assumption that we always consider finite continued fractions whose last digit is $1$ removes the case where $u=v=1$ from consideration.) Thus we want an estimation on the sum
\[
\sum_{\substack{uq_n+vq_{n-1}\le m\\ 1 \le v < u, \ (u,v)=1}} 1.
\]

To do this we will require a few estimations, which may be found in Chapters 1 and 2 of \cite{Sandor}. We are not using the strongest form of the estimation, merely the strongest form we need.
\begin{align*}
\sum_{\ell\le x} \frac{\phi(\ell)}{\ell}  &= \frac{6}{\pi^2} x + O\left( \log x \right)\\
\sum_{\ell\le x} \phi(\ell) &= \frac{3}{\pi^2} x^2 + O\left( x \log x\right)\\
\sum_{\ell\le x} d(\ell) &= O\left(x  \log x\right)
\end{align*}
In addition, we require the following result:
\[
\sum_{\substack{\ell \le x \\ (\ell,m)=1}} 1 = \sum_{d | m} \mu(d) \left\lfloor \frac{x}{d} \right\rfloor =x\sum_{d|m} \frac{\mu(d)}{d} + O(d(m))= \frac{\phi(m)}{m} x + O(d(m)).
\]

Thus we estimate the sum in the following way:
\begin{align*}
\sum_{\substack{uq_n+vq_{n-1}\le m\\ 1 \le v < u, \ (u,v)=1}} 1 &= \sum_{u\le m/q_n} \left( \sum_{\substack{v\le \min(u, (m-uq_n)/q_{n-1})\\(u,v)=1  }} 1\right)\\
&= \sum_{u\le m/q_n} \left( \frac{\phi(u)}{u} \cdot \min\left( u, \frac{m-uq_n}{q_{n-1}} \right) +O\left( d(u)\right)\right)\\
&= \sum_{u\le m/(q_n+q_{n-1})} \phi(u) + \sum_{m/(q_n+q_{n-1})<u\le m/q_n} \frac{\phi(u)}{u}\cdot \frac{m}{q_{n-1}}\\ &\qquad  -  \sum_{m/(q_n+q_{n-1})<u\le m/q_n}\frac{q_n}{q_{n-1}} \cdot \phi(u)   + O\left( \sum_{u\le m/q_n} d(u)\right)\\
&= \frac{3}{\pi^2} \left( \frac{m}{q_n+q_{n-1}}\right)^2 +O\left( \frac{m}{q_n+q_{n-1}}\log \left( \frac{m}{q_n+q_{n-1}}\right)\right) \\
&\qquad + \frac{6}{\pi^2} \cdot \frac{m}{q_{n-1}}\left( \frac{m}{q_n}-\frac{m}{q_n+q_{n-1}}\right)+ O\left( \frac{m}{q_{n-1}} \log \left( \frac{m}{q_n}\right)\right)\\
&\qquad - \frac{3}{\pi^2} \frac{q_n}{q_{n-1}} \left( \left( \frac{m}{q_n}\right)^2 -\left( \frac{m}{q_n+q_{n-1}}\right)^2  \right) +O\left( \frac{m}{q_{n-1}} \log\left( \frac{m}{q_n}\right)\right)\\
&\qquad + O \left( \frac{m}{q_n}\log \left( \frac{m}{q_n}\right)\right)\\
&= \frac{3}{\pi^2} \cdot \frac{m^2}{q_n(q_n+q_{n-1})} +O\left( \frac{m}{q_{n-1}} \log\left( \frac{m}{q_n}\right)\right).
\end{align*}
In the last step of this chain of equalities, the explicit terms on each side truly are equal to one another. By \eqref{eq:qnboundedbym}, we know that the main term in the last line above is $\gg m^{4\delta-2\delta/g+4\delta^2/g}$, while the big-Oh term is bounded by $\ll m^{2\delta+\delta/12g-\delta^2/6g} \log m$. Since $4-2/g \approx 2.314$ and $2+1/12g\approx 2.070$, the main term exceeds the big-Oh term by at least $m^{\delta/10}$. Therefore we have that
\[
\sum_{\substack{uq_n+vq_{n-1}\le m\\ 1 \le v < u, \ (u,v)=1}} 1 =  \frac{3}{\pi^2} \cdot \frac{m^2}{q_n(q_n+q_{n-1})} (1+O(m^{-\delta/10})),
\]
where $\delta/10$ is some positive constant dependent only on $\delta$. The big-Oh constant here is also uniform.

So the number of rationals (in lowest terms) with denominator at most $m$ in $\Gamma'_m$ must be 
\begin{align*}
\sum_{C_s}  \frac{3}{\pi^2} \cdot \frac{m^2}{q_n(q_n+q_{n-1})} (1+O(m^{-\delta/10}))& = \frac{3}{\pi^2} m^2 \lambda( \Gamma'_m) (1+O(m^{-\delta/10}))  \\
&= \frac{3}{\pi^2} m^2 \left(1+O\left(\frac{1}{\log m}\right)\right) (1+O(m^{-\delta/10}))\\
&=  \frac{3}{\pi^2} m^2 + O\left( \frac{m^2}{\log m}\right).
\end{align*}

Since the total number of rationals (in lowest terms) with denominator precisely $m$ in the interval $(0,1)$ is $\phi(m)$, it follows that the total number of rationals (in lowest terms) with denominator at most $m$ in the interval $(0,1)$ is \[ 
\frac{3}{\pi^2}m^2 + O(m\log m).\] Thus, the number of rationals (in lowest terms) with denominator at most $m$ not in $\Gamma'_m$ is $O(m^2/\log m)$. Since, as we noted before, every element of $\Gamma_m$ must have denominator at most $m$ but cannot be in $\Gamma'_m$, this implies that $|\Gamma_{m,\delta,s,\eta}| = O(m^2/\log m)$.
\end{proof}

\section{To $(\epsilon,s)$-normality}

Given $\epsilon>0$ and a non-empty finite string $s$, we will say that a rational number $r$ (with $L(r)$ continued fraction digits and lowest term denominator $q$) is $(\epsilon,s)$-normal if 
\begin{equation}\label{eq:esnormal1}
\left| \frac{A_{s}(r)}{L(r)} - \mu(C_s) \right| < \epsilon
\end{equation}
and
\begin{equation}\label{eq:esnormal2}
\left| \frac{\log q}{L(r)} - g \right| <\epsilon
\end{equation}
Here $A_s(r)$ is the number of time the string $s$ appears in the digits of $r$. 

\begin{prop}\label{prop:esnormal}Let $\epsilon>0$ and $s$ be fixed. 
The number of rational numbers with denominator at most $m$ that are \emph{not} $(\epsilon,s)$-normal is at most $O(m^2/\log m)$, with the implicit constant depending on $\epsilon$ and $s$.
\end{prop}

\begin{proof}
We claim that if $r\not\in \Gamma_{m,\eta,s,\delta}$ for some appropriate choice of $\eta$ and $\delta$, depending only on $\epsilon$,  then $r $ is $(\epsilon,s)$-normal. Again, let $n=n_\delta= \lfloor (1-2\delta)(\log m)/g\rfloor$. 

If $r\in \Gamma_{m,\eta,s,\delta}$ and $q$ is the  denominator of $r$, then 
\begin{align*}
A_s(r)-L(r) \mu(C_s)&= A_s(n-k+1;r)-n\mu(C_s) + O(L(r)-n)
&= O(\eta n) + O(L(r)-n)
\end{align*}
Following Lemma 9.6 in \cite{Bugeaud}, we note that in the proof of Proposition \ref{prop:mainestimate}, we have that if 
\[
r= \frac{up_n+vp_{n-1}}{uq_n+vq_{n-1}},
\]
then $m \ge q = uq_n+vq_{n-1}> uq_n$, so by applying \eqref{eq:qnboundedbym}, we have that $u \ll m^{3\delta}$. Since $v/u$ has $L(r)-n$ digits, we must have that $u\gg G^{L(r)-n}$ by \eqref{eq:equation}. Comparing these two bounds on $u$, we see that $L(r)-n \ll  \delta\log m\ll \delta n$.

Thus
\[
\left| A_s(r) - L(r) \mu(C_s) \right| = O((\eta+\delta) n) = O((\eta+\delta) L(r)).
\]
By choosing $\eta $ and $\delta$ sufficiently small in terms of $\epsilon$, we obtain \eqref{eq:esnormal1}.

For \eqref{eq:esnormal2}, if $r\in \Gamma_{m,\eta,s,\delta}$ and $q$ is the denominator of $r$, we have
\begin{align*}
\left| \log q - L(r) \cdot g \right| &\le \left| \log q - \log q_n \right| + \left| \log q_n - n \cdot g\right| + (L(r)-n)\cdot g\\
&\le \left| \log (m/q_n)\right| + \delta \cdot n+ O(\delta \log m) \\
&= O(\delta \cdot n) = O(\delta \cdot L(r)). 
\end{align*}
Again, by choosing $\delta$ small enough, we obtain the desired relation.
\end{proof}

\begin{prop}\label{prop:esnormal2}
Let the sequence $\{r_i\}_{i=1}^\infty$ be as in Theorem \ref{thm:main}, and let $\epsilon>0$ and $s$ be fixed. Then the number of $i\le x$ for which $r_i$ is not $(\epsilon,s)$-normal is $O(x/\log x)$.
\end{prop}

\begin{proof}
We assume without loss of generality that $x$ is a positive integer. Let $m$ be the denominator of $r_x$. Then, since the number of rationals in lowest terms with denominator $n$ is $\phi(n)$, we have that
\[
\sum_{n\le m-1} \phi(n) < x \le  \sum_{n\le m} \phi(n),
\]
or, by applying our earlier estimates on the sum of $\phi(n)$, we have
\[
\frac{3}{\pi^2} (m-1)^2 + O((m-1)\log(m-1)) < x \le \frac{3}{\pi^2} m^2+O(m\log m).
\]
By rearranging, we see that $x \asymp m^2$.

The number of $i\le x$ for which $r_i$ is not $(\epsilon,s)$-normal is at most the number of $r_i$'s with denominator at most $m$ that are not $(\epsilon,s)$-normal, and by Proposition \ref{prop:esnormal}, this is at msot $O(m^2/\log m)$. Since $x\asymp m^2$, this gives the desired result.
\end{proof}

\section{Proof of Theorem \ref{thm:main}}

Let $s$ be an arbitrary finite string of digits. To prove that $x_f$ is normal we must show that 
\[
\lim_{N\to \infty} \frac{A_s(N;x_f)}{N} = \mu(C_s).
\]
Let $\epsilon$ be an arbitrary positive number that will be allowed to go to $0$ at the end of the proof.

For a given integer $N$, let $M=M(N)$ be such that $N$th digit of $x$ lies in the string corresponding to the rational number $r_{f(M)}$, so that
\[
\sum_{n=1}^{M-1} L(r_{f(n)}) < N \le \sum_{n=1}^M L(r_{f(n)}).
\]
By one of our assumptions, we have that $L(r_{f(M)}) = o( \sum_{n=1}^M L(r_{f(n)}))$, so that $N\sim \sum_{n=1}^M L(r_{f(n)})$, $M=o(N)$, and $L(r_{f(M)})=o(N)$. Therefore, if we momentarily let $N'=\sum_{n=1}^{M} L(r_{f(n)})$, then
\[
A_s(N;x_f)= A_s(N';x_f) +o(N).
\]

The number of strings of length $L(s)$ that start in the expansion of one rational number $r_{f(n)}$ and end in the expansion of a different rational nmber $r_{f(n')}$, with $n,n'\le M$ is at most $m\cdot L(s)$. Since $L(s)$ is fixed and $M=o(N)$, we therefore have that
\[
A_s(N;x_f) =\sum_{n\le M} A_s(r_{f(n)})+ o(N).
\]

Let $S$ be the set of integers $n$ such that $r_{f(n)}$ is \emph{not} $(\epsilon,s)$-normal. The assumptions of the theorem together with Proposition \ref{prop:esnormal2} imply that $S$ has asymptotic density $0$. Therefore we have that
\begin{align*}
\sum_{\substack{n\le M \\ n \in S}} A_s(r_{f(n)})  &= O \left( \sum_{\substack{n\le M \\ n \in S}} L(r_{f(n)}) \right) = O\left( \max_{n\le M} L(r_{f(n)}) \cdot \sum_{\substack{n\le M \\ n \in S}} 1\right)\\
&= o \left( M \cdot   \max_{n\le M} L(r_{f(n)})\right) = o(N).
\end{align*}

Let $S^c = \mathbb{N}\setminus S$ denote the set of integers $n$ such that $r_{f(n)}$ is $(\epsilon,s)$-normal. In particular, if $n\in S^c$, we have that $A_s(r_{f(n)})=L(r_{f(n)}) (\mu(C_s)+O(\epsilon))$. Therefore,
\begin{align*}
\sum_{\substack{n\le M\\ n \in S^c}} A_s(r_{f(n)}) &= \sum_{\substack{n\le M\\ n \in S^c}}L(r_{f(n)}) (\mu(C_s)+O(\epsilon)) \\
&= \mu(C_s) \left(  \sum_{n \le M}L(r_{f(n)}) -  \sum_{\substack{n\le M\\ n \in S}}L(r_{f(n)})\right) + O\left( \epsilon  \sum_{\substack{n\le M\\ n \in S^c}}L(r_{f(n)})\right)\\
&= \mu(C_s) \left( N(1+o(1))-o(N)\right) + O\left( \epsilon \sum_{n\le M} L(r_{f(n)})\right)\\
&= \mu(C_s) N + o(N) + O(\epsilon N).
\end{align*}

Thus,
\[
A_s(N;x_f) = \mu(C_s)N + o(N) + O(\epsilon N).
\]
By dividing through by $N$ and noting that $\epsilon$ may be taken as small as desired, we get the desired equality in the limit.

\section{Proof of Corollary \ref{cor:main}}

This follows from Theorem \ref{thm:main}; we need only show that the conditions hold.

As in the statement of the corollary, let $\{i_j\}_{j=1}^\infty$ be the indices such that $r_{i_j}$ is in $R$, arranged in increasing order, so that the function $f(j)=i_j$ gives the desired function in Theorem \ref{thm:main}. The assumption on the size of $R(m)$, combined with the proof of Proposition \ref{prop:esnormal2}, shows that the number of $j$ such that $i_j\le x$ is an order of magnitude larger than $x/\log x$---that is,
\[
\frac{x}{\log x} = o\left( \#\{  i_j\le x \} \right).
\] Thus, if $S$ is any subset of $\mathbb{N}$ such that $\#\{n\in S:n\le x\}=O(x/\log x)$, then 
$\#\{n\in S:n\le x\}= o \left( \#\{i_j\le x\}\right)$. So, since the function $f$ is strictly increasing, we have that $\#f^{-1}(\{n\in S:n\le x\})= o(\# f^{-1}(\{i_j\le x\}))$, but since $f^{-1}(\{i_j\le x\})$ is just the set of all positive integers up to some point, this immediately implies that $f^{-1}(S)$ has asymptotic density $0$ as desired.

It remains to show the desired fact abouts about the $L$ function. We know that the maximum of $L(r)$ for $r\in R(m)$ is $O(\log m)$ (by \eqref{eq:equation}). At the same time, by \eqref{eq:esnormal2}, any rational number $r$ with denominator between $m^{1/2}$ and $m$ that is $(\epsilon,s)$-normal will have on the order of $\log m$, and by Proposition \ref{prop:esnormal}, this accounts for all but $O(m^2/\log m)$ of the rationals in $R(m)$, a negligible amount. Again applying the ideas of the proof of Proposition \ref{prop:esnormal2} to swap between considering all rationals with denominator at most $m$ and the first $x$ rationals, and then comparing these two facts with the necessary restrictions on $L(r)$ in Theorem \ref{thm:main} proves the corollary.

\section{Proof of Theorem \ref{thm:primes}}

Recall that we have three cases of subsequences we are considering in this case: \begin{enumerate}
\item the subsequence whose numerators are in $\mathbb{N}$ and whose denominators are in $\mathbb{P}$; 
\item the subsequence whose numerators are in $\mathbb{P}$ and whose denominators are in $\mathbb{N}$;  or,
\item the subsequence whose numerators and denominators are in $\mathbb{P}$.
\end{enumerate}
Let us refer to these rationals as Type 1, 2, or 3 rationals respectively, and denote the set of such rationals (in lowest terms) with denominator at most $m$ by $R_1(m)$, $R_2(m)$ or $R_3(m)$ respectively.

By elementary techniques, one can show that $R_1(m)$ and $R_2(m)$ are on the order of $m^2/\log m$ and $R_3(m)$ is on the order of  $m^2/(\log m)^2$. 

The only difference between the proof of this theorem and the proof of Corollary \ref{cor:main} will come in the estimates in the proof of Proposition \ref{prop:mainestimate}. There we considered strings $\frak{s}=[a_1,a_2,\dots,a_n]$ with two fractions $p_{n-1}/q_{n-1} = \langle a_1,a_2,\dots,a_{n-1}\rangle$ and $p_n /q_n = \langle a_1,a_2,\dots, a_n\rangle$. We showed that the number of fractions with denominator at most $m$ in the set $C_\frak{s}$ (with additional restrictions on $q_n$ and $q_{n-1}$) should be asymptotic to $\pi^2 m^2/6q_n(q_n+q_{n-1})$. Then, by summing over all intervals corresponding to ``good" strings, we obtained the desired result.

To prove Theorem \ref{thm:primes}, it suffices to prove that for any string $s$, we have that
\begin{equation}\label{eq:Csbound}
|C_\frak{s} \cap R_i(m) | = \begin{cases} O\left( \dfrac{m^2}{q_n(q_n+q_{n-1})\sqrt{\log m}} \right) & i =1,2\\
O\left( \dfrac{m^2}{q_n(q_n+q_{n-1})(\log m)^{3/2}} \right) & i=3.\end{cases}
\end{equation}
By summing over all $\frak{s}$'s of length $n$ with $C_\frak{s}$ \emph{not} in $\Gamma'_m$, we replace the $q_n(q_n+q_{n-1})$ term in the denominator with an additional copy of $\log m$, and therefore see that the number of fractions in $\Gamma_{m,\delta,s,\eta}\cap R_i(m)$  is little-oh of the number of fractions in $R_i(m)$. The rest of the proof is identical to that for Corollary \ref{cor:main}.

We remark briefly that in the proof of Theorem \ref{thm:main}, we had to sum over all $\frak{s}$'s with $C_\frak{s}$ in $\Gamma'_m$ because we needed stronger bounds on the size of the denominators. Otherwise, the size of the big-Oh term $O(m/q_{n-1})$ could overwhelm the size of the main term $m^2/q_n(q_n+q_{n-1})$. In this proof, as in the proof of Adler, Keane, and Smorodinsky, we can obtain sufficiently strong bounds on the size of the set without needing to assume anything good about the denominators, and thus can sum over the $C_\frak{s}$'s not in $\Gamma'_m$.

We have told a small lie above: we still need some bounds on the size of the denominators $q_n$. Consider $\frak{s}$ of length $n$ with $m\ge q_n \ge m\exp\{-(\log m)^{3/4}\}$. By our work earlier in this paper, the number of rational numbers with denominator at most $m$ in the cylinder set $C_{\frak{s}}$ is bounded by the number of positive integers $v\le u$ with $uq_n+vq_{n-1} \le m$. Clearly, we must have that $u\le m/q_n \le \exp\{(\log m)^{3/4}\}$, and $v$ must be bounded by $u$, so there are at most $\exp\{2(\log m)^{3/4}\} = o (m)$ such rationals. This is such an insignificant portion of the sets $R_i(m)$ that we may safely ignore them and presume that $q_n\le m\exp\{-(\log m)^{3/4}\}$ for the remainder of the proof.

We will need the following result.

\begin{lem}\label{lem:primes} Let $a,q,a',q'$ be positive integers with $(a,q)=(a',q')=1$.

Let $\pi'(x;q,a)$ denote the number of $\ell$ in the interval $1\le \ell \le x$ such that $\ell q+a$ is prime. Then 
\[
\pi'(x;q,a) \ll \frac{x \log \log 16q}{\log x}.
\]
Suppose that $aq'-qa'=t\neq 0$, and let $\pi'(x;q,a;q',a')$ denote the number of $\ell $ in the interval $1\le \ell \le x$ such that both $\ell q +a $ and $\ell q'+a'$ are both prime simultaneously. Then
\[
\pi'(x;q,a;q',a') \ll \frac{x (\log \log (16qq'))^2 \log \log 16|t|}{(\log x)^2}.
\]
The bounds in this lemma are uniform in all variables.
\end{lem}

\begin{proof}
We will apply Brun's sieve, following the work of Halberstam and Richert \cite{HR}. If we are considering $\pi'(x;q,a)$, let $\kappa=1$, and otherwise let $\kappa=2$. We let $Q=q$ if $\kappa=1$ and $Q=qq'$ otherwise. If $\kappa=1$, we will also assume that $t$ exists and equals $1$. We will also assume that $x$ is large enough so that $\log\log \log x$ exists and is positive.

We take $\mathcal{A}$ to be the set $[1,x]$, and, if $\kappa=1$, for a prime $p$,  we take $\mathcal{A}_p$ to be the subset of $\ell\in \mathcal{A}$ such that $\ell q+a$ is divisible by $p$. The set $\mathcal{A}_p$ is empty if $p|q$. Otherwise, $\mathcal{A}_p$ consists of the elements of $\mathcal{A}$ that fall into the reside class $-a/q$ modulo $p$. On the other hand, if $\kappa=2$, we now take $\mathcal{A}_p$ to be the subset of $\ell \in \mathcal{A}$ such that $\ell q +a$ or $\ell q'+a'$ is divisible by $p$. This set is again empty if either $p|q$ or $p|q'$. Otherwise, $\mathcal{A}_p$ consists of the elements of $\mathcal{A}$ that fall into the residue classes $-a/q$ or $-a'/q'$ modulo $p$. By assumption these are distinct unless $p$ divides $t$.

Then, we may apply Theorem 2.2 of Halberstam and Richert. (It is elementary to see that the conditions of the theorem hold, so we do not illustrate them here.) Therefore, the size of $\pi'(x;q,a)$ and $\pi'(x;q,a,q',a')$ is bounded  by
$O(x W(z))$,
with  $z=\sqrt{x}$. Here \[ W(z) = \prod_{\substack{p< z \\ p \nmid Q,\  p \nmid t}} \left( 1- \frac{\kappa}{p} \right) \cdot  \prod_{\substack{p< z \\ p \nmid Q,\  p | t}} \left( 1- \frac{1}{p} \right).\]

 Consider the sum $\sum_{p|Q} \frac{1}{p}$. This sum is maximized if all the primes dividing $Q$ are as small as possible. Since $\prod_{p\le y} p \ll e^y$ by the prime number theorem, we have that there exists a large constant $C$ so that
\[
\sum_{p|Q} \frac{1}{p}\ll \sum_{p\le C\log Q} \frac{1}{p} \ll \log \log \log  16Q,
\]
by Mertens' theorem. Here the $16$ is included to make sure that everything is positive. By a similar argument, one can show that
\[
\sum_{p|t} \frac{1}{p}\ll \log \log \log 16|t|
\]

Thus, we have that 
\begin{align*}
W(z) &\le \exp \left( - \sum_{p<z } \frac{\kappa}{p} + \sum_{p|Q} \frac{\kappa}{p} + \sum_{p|t} \frac{1}{p} \right) \\
&\ll \left( \frac{\log \log 16Q}{\log z} \right)^\kappa \cdot \log\log(16|t|).
\end{align*}
Recalling our assumption that $z=\sqrt{x}$, we obtain the desired bounds.
\end{proof}

Consider the case of the Type 1 rationals. Let $P_m$ denote all the primes less than $m$. Recall that we have assumed $q_n \le m \exp\{-(\log m)^{3/4}\}$ so that $(\log m/q_n)^{-1} \le (\log m)^{-3/4}$; also, trivially $\log \log 16q_n \ll \log \log m$. By applying Lemma \ref{lem:primes}, the number of points in $R_1(m)$ that are also an interval $C_\frak{s}$ is given by
\begin{align*}
\sum_{\substack{uq_n+vq_{n-1} \in P_m\\ 1\le v< u}} 1 &\le  \sum_{\substack{v\le m/q_n+q_{n-1} \\ (v,q_n)=1}} \pi'(m/q_n; q_n, vq_{n-1})\\
&\ll \sum_{\substack{v\le m/q_n+q_{n-1} \\ (v,q_n)=1}} \frac{m(\log \log 16q_n)^2}{q_n \log (m/q_n)}\\
&\ll \sum_{v\le m/q_n+q_{n-1} } \frac{m(\log \log m)^2}{q_n (\log m)^{3/4}}\\
&\ll \frac{m^2(\log \log m)^2}{q_n(q_n+q_{n-1})(\log m)^{3/4}},
\end{align*}
and this clearly satisfies \eqref{eq:Csbound}. Note that the restriction that $(u,v)=1$ is unnecessary due to wanting $uq_n+vq_{n-1}$ to be prime.

For the Type 2 rationals,  if we run through the argument the same way we did for the Type 1 rationals, we get the inequalities, but with $\pi'(m/q_n;p_n,vp_{n-1})$ in place of $\pi'(m/q_n;q_n,vq_{n-1})$. The desired bound follows in the same way.

The case of Type 3 rationals also proceeds as the case of Type 1 rationals, but with 
\[
\pi'(m/q_n; q_n, vq_{n-1}, p_n,vp_{n-1}) \text{ in place of } \pi' (m/q_n,q_n,vq_{n-1}).
\]
Recall that $|q_np_{n-1}-q_{n-1}p_n|=1$, so we have that $t$ in this case will equal $\pm v$, however, as we have that $v\le m/q_n+q_{n-1}$, we have that $\log \log 16|v|$ is bounded by $O(\log \log m)$, and the desired result holds from this.

This completes the proof.

\section{Acknowledgments}

The author acknowledges assistance from the Research and Training Group grant DMS-1344994 funded by the National Science Foundation.

The author would also like to thank Paul Pollack for his help.

\end{document}